\documentclass[11pt,a4paper]{amsart}
\usepackage{mathrsfs}
\usepackage{syntonly}
\usepackage{amsmath}
\usepackage{amsthm}
\usepackage{amsfonts}
\usepackage{amssymb}
\usepackage{latexsym}
\usepackage{amscd,amssymb,amsopn,amsmath,amsthm,graphics,amsfonts,mathrsfs,accents,enumerate,verbatim,calc}
\usepackage[dvips]{graphicx}
\usepackage[colorlinks=true,linkcolor=red,citecolor=blue]{hyperref}
\usepackage[all]{xy}
\usepackage{enumitem}

\date{}
\allowdisplaybreaks[4] \footskip=15pt
\renewcommand{\uppercasenonmath}[1]{}

\numberwithin{equation}{section} \theoremstyle{plain}
\newtheorem{lem}{Lemma}[section]
\newtheorem{cor}[lem]{Corollary}
\newtheorem{prop}[lem]{Proposition}
\newtheorem{thm}[lem]{Theorem}

\newtheorem{cond}[lem]{Condition}
\newtheorem{definition}[lem]{Definition}
\newtheorem{Ex}[lem]{Example}
\newtheorem{Quest}[lem]{Question}
\newtheorem{Property}[lem]{Property}
\newtheorem{Properties}[lem]{Properties}
\newtheorem{Subprops}{}[lem]
\newtheorem{Para}[lem]{}

\newtheorem{rem}[lem]{Remark}

\newtheorem*{ack*}{ACKNOWLEDGEMENTS}

\newcommand{\pf}{\noindent\begin {proof}}
\newcommand{\epf}{\end{proof}}

\newcommand{\ra}{\rightarrow}

\begin{document}
\begin{center}
{\Large  \bf Avramov-Martsinkovsky type exact sequences for extriangulated categories}

\vspace{0.5cm}  Jiangsheng Hu,  Dongdong Zhang, Tiwei Zhao\footnote{Corresponding author. \\ Jiangsheng Hu was supported by the NSF of China (Grants Nos. 11671069, 11771212), Qing Lan Project of Jiangsu Province and Jiangsu Government Scholarship for Overseas Studies (JS-2019-328). Tiwei Zhao was supported by the NSF of China (Grants Nos. 11971225, 11901341), the project ZR2019QA015 supported by Shandong Provincial Natural Science Foundation, and the Young Talents Invitation Program of Shandong Province. Panyue Zhou was supported by the National Natural Science Foundation of China (Grant Nos. 11901190, 11671221),  the Hunan Provincial Natural Science Foundation of China (Grant No. 2018JJ3205) and  the Scientific Research Fund of Hunan Provincial Education Department (Grant No. 19B239).} and Panyue Zhou
\end{center}
\medskip
\medskip
\bigskip
\centerline { \bf  Abstract}
\medskip
\leftskip10truemm \rightskip10truemm \noindent
\hspace{1em}Let $(\mathcal{C},\mathbb{E},\mathfrak{s})$ be an extriangulated category with a proper class $\xi$ of
$\mathbb{E}$-triangles. In this paper, we first introduce the $\xi$-Gorenstein cohomology in terms of $\xi$-$\mathcal{G}$projective resolutions and $\xi$-$\mathcal{G}$injective coresolutions, respectively, and then we get the balance of $\xi$-Gorenstein cohomology. Moreover, we study the interplay among $\xi$-cohomology, $\xi$-Gorenstein cohomology and $\xi$-complete cohomology, and obtain the Avramov-Martsinkovsky type exact sequences in this setting.  \\[2mm]

\leftskip0truemm \rightskip0truemm
\section { \bf Introduction}

Avramov and Martsinkovsky \cite{AM}  introduced relative and Tate cohomology theories for modules
of finite $G$-dimension, which were initially defined for representations of finite groups. They
made an intensive study of the interaction between the absolute, relative and Tate cohomology
theories. More precisely, they showed that absolute cohomology,
Gorenstein cohomology and Tate cohomology can be connected by a long exact sequence (see \cite[Theorem 7.1]{AM}). Ever since then several authors have studied these theories in different abelian categories (see \cite{Asadollahi2008,CJ2014,Iacob2005,ST,Zhang2017} for instance).

Beligiannis developed in \cite{Bel1} a relative version of homological algebra in triangulated categories in analogy to relative homological algebra in abelian categories, in which the notion of a proper class of exact sequences is replaced by that of a proper class of triangles. By specifying a class of triangles $\mathcal{E}$, which is called a proper class of triangles, he introduced $\mathcal{E}$-projective and $\mathcal{E}$-injective objects. In an attempt to extend the theory, Asadollahi and Salarian \cite{AS1} introduced and studied $\mathcal{E}$-Gorenstein projective,
$\mathcal{E}$-Gorenstein injective objects, and corresponding $\mathcal{E}$-Gorenstein
dimensions in triangulated categories by modifying what Enochs, Jenda \cite{EJ2}
and Holm \cite{hG} have done in the category of modules. Moreover, Tate cohomology theory in a triangulated category was developed in \cite{AS2}. Ren and Liu established  the global $\xi$-Gorenstein dimension for a triangulated category in \cite{RL1} by introducing
$\mathcal{E}$-Gorenstein cohomology groups $\mathcal{E}xt^{i}_{\mathcal{GP}}(-,-)$
 and $\mathcal{E}xt^{i}_{\mathcal{GI}}(-,-)$ for objects with finite
$\mathcal{E}$-Gorenstein dimension. Motivated by Avramov-Martsinkovsky type exact sequences constructed over a ring $R$ in \cite{AM}, Ren, Zhao and Liu \cite{RL3} proved that Beligiannis's $\mathcal{E}$-cohomology, Asadallahi
and Salarian's $\mathcal{E}$-Tate cohomology and Ren and Liu's Gorenstein cohomology can be connected by a long exact sequence.

The notion of extriangulated categories was introduced by Nakaoka and Palu in \cite{NP} as a simultaneous generalization of
exact categories and triangulated categories. Exact categories and extension closed subcategories of an
extriangulated category are extriangulated categories, while there exist some other examples of extriangulated categories which are neither exact nor triangulated, see \cite{HZZ,LN,NP,ZZ}. Hence many results on exact categories
and triangulated categories can be unified in the same framework.

Let $(\mathcal{C},\mathbb{E},\mathfrak{s})$  be an extriangulated category with a proper class $\xi$ of $\mathbb{E}$-triangles. The authors \cite{HZZ} studied a relative homological algebra in $\mathcal C$ which parallels the relative homological algebra in a triangulated category. By specifying a class of $\mathbb{E}$-triangles, which is called a proper class $\xi$ of
$\mathbb{E}$-triangles, the authors introduced $\xi$-projective dimensions and $\xi$-$\mathcal{G}$projective dimensions,
and discussed their properties. Recently, we studied $\xi$-cohomology in \cite{HZZ1} and developed a $\xi$-complete cohomology theory for an extriangulated category in \cite{HZZZ}, which extends Tate cohomology defined in the category of modules or in a triangulated category. The aim of this paper is to study Avramov-Martsinkovsky type exact sequences for extriangulated categories.

We now outline the results of the paper. In Section 2, we summarize some preliminaries
and basic facts about extriangulated categories which will be used throughout the paper.

From Section 3, we assume that $(\mathcal{C}, \mathbb{E}, \mathfrak{s})$ is an extriangulated category with enough $\xi$-projectives and enough $\xi$-injectives satisfying Condition (WIC). We first introduce $\xi$-Gorenstein cohomology in terms of $\xi$-$\mathcal{G}$projective resolutions and $\xi$-$\mathcal{G}$injective coresolutions, and then  prove that $\xi$-Gorenstein cohomology in $(\mathcal{C}, \mathbb{E}, \mathfrak{s})$ is
balanced (see Theorem \ref{thm:balance-of-gorenstein}).  Moreover, we show that
 there are two long exact sequences of $\xi$-Gorenstein cohomology under some certain conditions (see Propositions \ref{prop:long-exact1} and \ref{prop:long-exact2}).

In Section 4, we first recall some definitions and basic properties of $\xi$-complete cohomology in $(\mathcal{C}, \mathbb{E}, \mathfrak{s})$, and then  construct the Avramov-Martsinkovsky type
exact sequence in $(\mathcal{C}, \mathbb{E}, \mathfrak{s})$. More precisely, it is proved that $\xi$-cohomology, $\xi$-Gorenstein cohomology and $\xi$-complete cohomology can be connected by a long exact sequence, which generalizes Avramov-Martsinkovsky's result on a category of modules
and Ren-Zhao-Liu's result on a triangulated category and is new for  exact categories and extension-closed subcategories of triangulated categories
 (see Theorem \ref{AM} and Remark \ref{rem:4.5}).

\section{\bf Preliminaries}
Throughout this paper, we always assume that  $\mathcal{C}=(\mathcal{C}, \mathbb{E}, \mathfrak{s})$ is an extriangulated category and $\xi$ is a proper class of $\mathbb{E}$-triangles in  $\mathcal{C}$.  We also assume that the extriangulated category $\mathcal{C}$ has enough $\xi$-projectives and enough $\xi$-injectives satisfying Condition (WIC). Next we briefly recall some definitions and basic properties of extriangulated categories from \cite{NP}.
We omit some details here, but the reader can find them in \cite{NP}.

Let $\mathcal{C}$ be an additive category equipped with an additive bifunctor
$$\mathbb{E}: \mathcal{C}^{\rm op}\times \mathcal{C}\rightarrow {\rm Ab},$$
where ${\rm Ab}$ is the category of abelian groups. For any objects $A, C\in\mathcal{C}$, an element $\delta\in \mathbb{E}(C,A)$ is called an $\mathbb{E}$-extension.
Let $\mathfrak{s}$ be a correspondence which associates an equivalence class $$\mathfrak{s}(\delta)=\xymatrix@C=0.8cm{[A\ar[r]^x
 &B\ar[r]^y&C]}$$ to any $\mathbb{E}$-extension $\delta\in\mathbb{E}(C, A)$. This $\mathfrak{s}$ is called a {\it realization} of $\mathbb{E}$, if it makes the diagrams in \cite[Definition 2.9]{NP} commutative.
 A triplet $(\mathcal{C}, \mathbb{E}, \mathfrak{s})$ is called an {\it extriangulated category} if it satisfies the following conditions.
\begin{enumerate}
\item $\mathbb{E}\colon\mathcal{C}^{\rm op}\times \mathcal{C}\rightarrow \rm{Ab}$ is an additive bifunctor.

\item $\mathfrak{s}$ is an additive realization of $\mathbb{E}$.

\item $\mathbb{E}$ and $\mathfrak{s}$  satisfy the compatibility conditions in \cite[Definition 2.12]{NP}.

 \end{enumerate}

\begin{rem}
Note that both exact categories and triangulated categories are extriangulated categories $($see \cite[Example 2.13]{NP}$)$ and extension closed subcategories of extriangulated categories are
again extriangulated $($see \cite[Remark 2.18]{NP}$)$. Moreover, there exist extriangulated categories which
are neither exact categories nor triangulated categories $($see \cite[Proposition 3.30]{NP}, \cite[Example 4.14]{ZZ} and \cite[Remark 3.3]{HZZ}$)$.
\end{rem}

We will use the following terminology.

\begin{definition}{ \emph{(see \cite[Definitions 2.15 and 2.19]{NP})}} {\rm
 Let $(\mathcal{C}, \mathbb{E}, \mathfrak{s})$ be an extriangulated category.
\begin{enumerate}
\item A sequence $\xymatrix@C=1cm{A\ar[r]^x&B\ar[r]^{y}&C}$ is called a {\it conflation} if it realizes some $\mathbb{E}$-extension $\delta\in\mathbb{E}(C, A)$.
In this case, $x$ is called an {\it inflation} and $y$ is called a {\it deflation}.

\item  If a conflation $\xymatrix@C=0.6cm{A\ar[r]^x&B\ar[r]^{y}&C}$ realizes $\delta\in\mathbb{E}(C, A)$, we call the pair
$\xymatrix@C=0.6cm{(A\ar[r]^x&B\ar[r]^{y}&C, \delta)}$ an {\it $\mathbb{E}$-triangle}, and write it in the following.
\begin{center} $\xymatrix{A\ar[r]^x&B\ar[r]^{y}&C\ar@{-->}[r]^{\delta}&}$\end{center}
We usually do not write this ``$\delta$" if it is not used in the argument.

\item Let $\xymatrix{A\ar[r]^x&B\ar[r]^{y}&C\ar@{-->}[r]^{\delta}&}$ and $\xymatrix{A'\ar[r]^{x'}&B'\ar[r]^{y'}&C'\ar@{-->}[r]^{\delta'}&}$
be any pair of $\mathbb{E}$-triangles. If a triplet $(a, b, c)$ realizes $(a, c): \delta\rightarrow \delta'$, then we write it as
 $$\xymatrix{A\ar[r]^{x}\ar[d]_{a}&B\ar[r]^{y}\ar[d]_{b}&C\ar[d]_{c}\ar@{-->}[r]^{\delta}&\\
 A'\ar[r]^{x'}&B'\ar[r]^{y'}&C'\ar@{-->}[r]^{\delta'}&}$$
 and call $(a, b, c)$ a {\it morphism} of $\mathbb{E}$-triangles.
\end{enumerate}}

\end{definition}

The following condition is analogous to the weak idempotent completeness in exact category (see \cite[Condition 5.8]{NP}).

\begin{cond} \label{cond:4.11} \emph{({\rm Condition (WIC)})}  Consider the following conditions.
\begin{enumerate}
\item  Let $f\in\mathcal{C}(A, B), g\in\mathcal{C}(B, C)$ be any composable pair of morphisms. If $gf$ is an inflation, then so is $f$.

\item Let $f\in\mathcal{C}(A, B), g\in\mathcal{C}(B, C)$ be any composable pair of morphisms. If $gf$ is a deflation, then so is $g$.

\end{enumerate}
\end{cond}

\begin{Ex}\label{Ex:4.12}

\emph{(1)} If $\mathcal{C}$ is an exact category, then Condition \emph{(WIC)} is equivalent to that $\mathcal{C}$ is
weakly idempotent complete \emph{(see \cite[Proposition 7.6]{B"u})}.

\emph{(2)} If $\mathcal{C}$ is a triangulated category, then Condition \emph{(WIC)} is automatically satisfied.
\end{Ex}

\begin{lem}\label{lem1} \emph{(see \cite[Proposition 3.15]{NP})} Assume that $(\mathcal{C}, \mathbb{E},\mathfrak{s})$ is an extriangulated category. Let $C$ be any object, and let $\xymatrix@C=2em{A_1\ar[r]^{x_1}&B_1\ar[r]^{y_1}&C\ar@{-->}[r]^{\delta_1}&}$ and $\xymatrix@C=2em{A_2\ar[r]^{x_2}&B_2\ar[r]^{y_2}&C\ar@{-->}[r]^{\delta_2}&}$ be any pair of $\mathbb{E}$-triangles. Then there is a commutative diagram
in $\mathcal{C}$
$$\xymatrix{
    & A_2\ar[d]_{m_2} \ar@{=}[r] & A_2 \ar[d]^{x_2} \\
  A_1 \ar@{=}[d] \ar[r]^{m_1} & M \ar[d]_{e_2} \ar[r]^{e_1} & B_2\ar[d]^{y_2} \\
  A_1 \ar[r]^{x_1} & B_1\ar[r]^{y_1} & C   }
  $$
  which satisfies $\mathfrak{s}(y^*_2\delta_1)=\xymatrix@C=2em{[A_1\ar[r]^{m_1}&M\ar[r]^{e_1}&B_2]}$ and
  $\mathfrak{s}(y^*_1\delta_2)=\xymatrix@C=2em{[A_2\ar[r]^{m_2}&M\ar[r]^{e_2}&B_1]}$.
\end{lem}

The following definitions are quoted verbatim from \cite[Section 3]{HZZ}. A class of $\mathbb{E}$-triangles $\xi$ is {\it closed under base change} if for any $\mathbb{E}$-triangle $$\xymatrix@C=2em{A\ar[r]^x&B\ar[r]^y&C\ar@{-->}[r]^{\delta}&\in\xi}$$ and any morphism $c\colon C' \to C$,  any $\mathbb{E}$-triangle  $\xymatrix@C=2em{A\ar[r]^{x'}&B'\ar[r]^{y'}&C'\ar@{-->}[r]^{c^*\delta}&}$ belongs to $\xi$.

Dually, a class of  $\mathbb{E}$-triangles $\xi$ is {\it closed under cobase change} if for any $\mathbb{E}$-triangle $$\xymatrix@C=2em{A\ar[r]^x&B\ar[r]^y&C\ar@{-->}[r]^{\delta}&\in\xi}$$ and any morphism $a\colon A \to A'$,  any $\mathbb{E}$-triangle  $\xymatrix@C=2em{A'\ar[r]^{x'}&B'\ar[r]^{y'}&C\ar@{-->}[r]^{a_*\delta}&}$ belongs to $\xi$.

A class of $\mathbb{E}$-triangles $\xi$ is called {\it saturated} if in the situation of Lemma \ref{lem1}, whenever  \\
$\xymatrix@C=2em{A_2\ar[r]^{x_2}&B_2\ar[r]^{y_2}&C\ar@{-->}[r]^{\delta_2 }&}$
 and $\xymatrix@C=2em{A_1\ar[r]^{m_1}&M\ar[r]^{e_1}&B_2\ar@{-->}[r]^{y_2^{\ast}\delta_1}&}$
 belong to $\xi$, then the  $\mathbb{E}$-triangle $$\xymatrix@C=2em{A_1\ar[r]^{x_1}&B_1\ar[r]^{y_1}&C\ar@{-->}[r]^{\delta_1 }&}$$  belongs to $\xi$.

An $\mathbb{E}$-triangle $\xymatrix@C=2em{A\ar[r]^x&B\ar[r]^y&C\ar@{-->}[r]^{\delta}&}$ is called {\it split} if $\delta=0$. It is easy to see that it is split if and only if $x$ is section or $y$ is retraction. The full subcategory  consisting of the split $\mathbb{E}$-triangles will be denoted by $\Delta_0$.

  \begin{definition} \emph{(see \cite[Definition 3.1]{HZZ})}\label{def:proper class} {\rm  Let $\xi$ be a class of $\mathbb{E}$-triangles which is closed under isomorphisms. Then $\xi$ is called a {\it proper class} of $\mathbb{E}$-triangles if the following conditions hold:

  \begin{enumerate}
\item  $\xi$ is closed under finite coproducts and $\Delta_0\subseteq \xi$.

\item $\xi$ is closed under base change and cobase change.

\item $\xi$ is saturated.

  \end{enumerate}}
  \end{definition}

 \begin{definition} \emph{(see \cite[Definition 4.1]{HZZ})}
 {\rm An object $P\in\mathcal{C}$  is called {\it $\xi$-projective}  if for any $\mathbb{E}$-triangle $\xymatrix{A\ar[r]^x& B\ar[r]^y& C \ar@{-->}[r]^{\delta}& }$ in $\xi$, the induced sequence of abelian groups $$\xymatrix@C=0.6cm{0\ar[r]& \mathcal{C}(P,A)\ar[r]& \mathcal{C}(P,B)\ar[r]&\mathcal{C}(P,C)\ar[r]& 0}$$ is exact. Dually, we have the definition of {\it $\xi$-injective} objects.}
\end{definition}

We denote by $\mathcal{P(\xi)}$ (resp. $\mathcal{I(\xi)}$) the class of $\xi$-projective (resp. $\xi$-injective) objects of $\mathcal{C}$. It follows from the definition that this subcategory $\mathcal{P}(\xi)$ and $\mathcal{I}(\xi)$ are full, additive, closed under isomorphisms and direct summands.

 An extriangulated  category $(\mathcal{C}, \mathbb{E}, \mathfrak{s})$ is said to  have {\it  enough
$\xi$-projectives} \ (resp. {\it  enough $\xi$-injectives}) provided that for each object $A$ there exists an $\mathbb{E}$-triangle $\xymatrix@C=2.1em{K\ar[r]& P\ar[r]&A\ar@{-->}[r]& }$ (resp. $\xymatrix@C=2em{A\ar[r]& I\ar[r]& K\ar@{-->}[r]&}$) in $\xi$ with $P\in\mathcal{P}(\xi)$ (resp. $I\in\mathcal{I}(\xi)$).

The {\it $\xi$-projective dimension} $\xi$-${\rm pd} A$ of $A\in\mathcal{C}$ is defined inductively.
 If $A\in\mathcal{P}(\xi)$, then define $\xi$-${\rm pd} A=0$.
Next if $\xi$-${\rm pd} A>0$, define $\xi$-${\rm pd} A\leq n$ if there exists an $\mathbb{E}$-triangle
 $K\to P\to A\dashrightarrow$  in $\xi$ with $P\in \mathcal{P}(\xi)$ and $\xi$-${\rm pd} K\leq n-1$.
Finally we define $\xi$-${\rm pd} A=n$ if $\xi$-${\rm pd} A\leq n$ and $\xi$-${\rm pd} A\nleq n-1$. Of course we set $\xi$-${\rm pd} A=\infty$, if $\xi$-${\rm pd} A\neq n$ for all $n\geq 0$.

Dually we can define the {\it $\xi$-injective dimension}  $\xi$-${\rm id} A$ of an object $A\in\mathcal{C}$.

We denote by $\widetilde{{\mathcal{P}}}(\xi)$ (resp. $\widetilde{{\mathcal{I}}}(\xi)$) the full subcategory of $\mathcal{C}$ whose objects have finite $\xi$-projective (resp. $\xi$-injective) dimension.

\begin{definition} \emph{(see \cite[Definition 4.4]{HZZ})}
{\rm A {\it $\xi$-exact} complex $\mathbf{X}$ is a diagram $$\xymatrix@C=2em{\cdots\ar[r]&X_1\ar[r]^{d_1}&X_0\ar[r]^{d_0}&X_{-1}\ar[r]&\cdots}$$ in $\mathcal{C}$ such that for each integer $n$, there exists an $\mathbb{E}$-triangle $\xymatrix@C=2em{K_{n+1}\ar[r]^{g_n}&X_n\ar[r]^{f_n}&K_n\ar@{-->}[r]^{\delta_n}&}$ in $\xi$ and $d_n=g_{n-1}f_n$.
}\end{definition}

\begin{definition} \emph{(see \cite[Definition 4.5]{HZZ})}
{\rm Let $\mathcal{W}$ be a class of objects in $\mathcal{C}$. An $\mathbb{E}$-triangle
$$\xymatrix@C=2em{A\ar[r]& B\ar[r]& C\ar@{-->}[r]& }$$ in $\xi$ is said to be
{\it $\mathcal{C}(-,\mathcal{W})$-exact} (resp.
{\it $\mathcal{C}(\mathcal{W},-)$-exact}) if for any $W\in\mathcal{W}$, the induced sequence of abelian groups
\begin{gather*}
  \xymatrix@C=2em{0\ar[r]&\mathcal{C}(C,W)\ar[r]&\mathcal{C}(B,W)\ar[r]&\mathcal{C}(A,W)\ar[r]& 0} \\
  (\mbox{resp. } \xymatrix@C=2em{0\ar[r]&\mathcal{C}(W,A)\ar[r]&\mathcal{C}(W,B)\ar[r]&\mathcal{C}(W,C)\ar[r]&0})
\end{gather*}
 is exact in ${\rm Ab}$}.
\end{definition}

\begin{definition} \emph{(see \cite[Definition 4.6]{HZZ})}
 {\rm Let $\mathcal{W}$ be a class of objects in $\mathcal{C}$. A complex $\mathbf{X}$ is called {\it $\mathcal{C}(-,\mathcal{W})$-exact} (resp.
{\it $\mathcal{C}(\mathcal{W},-)$-exact}) if it is a $\xi$-exact complex
$$\xymatrix@C=2em{\cdots\ar[r]&X_1\ar[r]^{d_1}&X_0\ar[r]^{d_0}&X_{-1}\ar[r]&\cdots}$$ in $\mathcal{C}$ such that  there is a $\mathcal{C}(-,\mathcal{W})$-exact (resp.
 $\mathcal{C}(\mathcal{W},-)$-exact) $\mathbb{E}$-triangle $$\xymatrix@C=2em{K_{n+1}\ar[r]^{g_n}&X_n\ar[r]^{f_n}&K_n\ar@{-->}[r]^{\delta_n}&}$$ in $\xi$ for each integer $n$ and $d_n=g_{n-1}f_n$.

 A $\xi$-exact complex $\mathbf{X}$ is called {\it complete $\mathcal{P}(\xi)$-exact} (resp. {\it complete $\mathcal{I}(\xi)$-exact}) if it is $\mathcal{C}(-,\mathcal{P}(\xi))$-exact (resp.
 $\mathcal{C}(\mathcal{I}(\xi),-)$-exact).}
\end{definition}

\begin{definition} \emph{(see \cite[Definition 4.7]{HZZ})}
 {\rm A  {\it complete $\xi$-projective resolution}  is a complete $\mathcal{P}(\xi)$-exact complex\\ $$\xymatrix@C=2em{\mathbf{P}:\cdots\ar[r]&P_1\ar[r]^{d_1}&P_0\ar[r]^{d_0}&P_{-1}\ar[r]&\cdots}$$ in $\mathcal{C}$ such that $P_n$ is $\xi$-projective for each integer $n$. Dually,   a  {\it complete $\xi$-injective coresolution}  is a complete $\mathcal{I}(\xi)$-exact complex $$\xymatrix@C=2em{\mathbf{I}:\cdots\ar[r]&I_1\ar[r]^{d_1}&I_0\ar[r]^{d_0}&I_{-1}\ar[r]&\cdots}$$ in $\mathcal{C}$ such that $I_n$ is $\xi$-injective for each integer $n$.}
\end{definition}

\begin{definition} \emph{(see \cite[Definition 4.8]{HZZ})}
{\rm  Let $\mathbf{P}$ be a complete $\xi$-projective resolution in $\mathcal{C}$. So for each integer $n$, there exists a $\mathcal{C}(-, \mathcal{P}(\xi))$-exact $\mathbb{E}$-triangle $\xymatrix@C=2em{K_{n+1}\ar[r]^{g_n}&P_n\ar[r]^{f_n}&K_n\ar@{-->}[r]^{\delta_n}&}$ in $\xi$. The objects $K_n$ are called {\it $\xi$-$\mathcal{G}$projective} for each integer $n$. Dually if  $\mathbf{I}$ is a complete $\xi$-injective  coresolution in $\mathcal{C}$, there exists a  $\mathcal{C}(\mathcal{I}(\xi), -)$-exact $\mathbb{E}$-triangle $\xymatrix@C=2em{K_{n+1}\ar[r]^{g_n}&I_n\ar[r]^{f_n}&K_n\ar@{-->}[r]^{\delta_n}&}$ in $\xi$ for each integer $n$. The objects $K_n$ are called {\it $\xi$-$\mathcal{G}$injective} for each integer $n$.}
\end{definition}

 We denote by $\mathcal{GP}(\xi)$ (resp. $\mathcal{GI}(\xi)$) the class of $\xi$-$\mathcal{G}$projective (resp. $\xi$-$\mathcal{G}$injective) objects.
It is obvious that $\mathcal{P(\xi)}$ $\subseteq$ $\mathcal{GP}(\xi)$ and $\mathcal{I(\xi)}$ $\subseteq$ $\mathcal{GI}(\xi)$.

\begin{definition} \emph{(see \cite[Definition 3.1]{HZZ1})}\label{df:resolution} {\rm Let $M$ be an object in $\mathcal{C}$. A {\it $\xi$-projective resolution} of $M$ is a $\xi$-exact complex $\mathbf{P}\rightarrow M$ such that $\mathbf{P}_n\in{\mathcal{P}(\xi)}$ for all $n\geq0$. Dually, a {\it $\xi$-injective coresolution} of $M$ is a $\xi$-exact complex $ M\rightarrow \mathbf{I}$ such that $\mathbf{I}_n\in{\mathcal{I}(\xi)}$ for all $n\leq0$.}
\end{definition}

\begin{definition} \emph{(see \cite[Definition 3.2]{HZZ1})}\label{df:derived-functors} {\rm Let $M$ and $N$ be objects in $\mathcal{C}$.
\begin{enumerate}
\item[{\rm (1)}] If we choose a $\xi$-projective resolution $\xymatrix@C=2em{\mathbf{P}\ar[r]& M}$ of  $M$, then for any integer $n\geq 0$, the \emph{$\xi$-cohomology group} $\xi{\rm xt}_{\mathcal{P}(\xi)}^n(M,N)$ are defined as
$$\xi{\rm xt}_{\mathcal{P}(\xi)}^n(M,N)=H^n({\mathcal{C}}(\mathbf{P},N)).$$
\item[{\rm (2)}] If we choose a
$\xi$-injective coresolution $\xymatrix@C=2em{N\ar[r]&\mathbf{I}}$ of  $N$, then for any integer $n\geq 0$, the \emph{$\xi$-cohomology group} $\xi{\rm xt}_{\mathcal{I}(\xi)}^n(M,N)$ are defined as $$\xi{\rm xt}_{\mathcal{I}(\xi)}^n(M,N)=H^n({\mathcal{C}}(M, \mathbf{I})).$$
\end{enumerate}}
\end{definition}

\begin{rem}\label{fact:2.5'} { ${\rm \xi xt}_{\mathcal{P}(\xi)}^n(-,-)$ and ${\rm \xi xt}_{\mathcal{I}(\xi)}^n(-,-)$ are cohomological functors for any integer $n\geq 0$, independent of the choice of $\xi$-projective resolutions and $\xi$-injective coresolutions, respectively. In fact, with the modifications of the usual proof, one obtains the isomorphism $\xi{\rm xt}_{\mathcal{P}(\xi)}^n(M,N)\cong \xi{\rm xt}_{\mathcal{I}(\xi)}^n(M,N),$
which is denoted by $\xi{\rm xt}_{\xi}^n(M,N).$
}
\end{rem}

\section{\bf $\xi$-Gorenstein cohomology}

Let $M\in\mathcal{C}$ and $\xymatrix@C=0.5cm{K\ar[r]&G\ar[r]^f&M\ar@{-->}[r]&}$ be an $\mathbb{E}$-triangle. We call the morphism $f$ a $\xi$-$\mathcal{G}$projective precover of $M$ if  $G\in\mathcal{GP}(\xi)$ and this $\mathbb{E}$-triangle is $\mathcal{C}(\mathcal{GP}(\xi),-)$-exact.

Let $M\in\mathcal{C}$. A $\xi$-exact complex $\mathbf{G}\to M$: $$\cdots\to G_2\to G_1\to G_0\to M\to 0$$ is called a \emph{$\xi$-$\mathcal{G}$projective resolution} of
$M$ if each $f_i$ is a $\xi$-$\mathcal{G}$projective precover of $K_i$ in the relevant $\mathbb{E}$-triangle $\xymatrix@C=0.5cm{K_{i+1}\ar[r]&G_i\ar[r]^{f_i}&K_i\ar@{-->}[r]&}$ (with $K_0=M$) for $i\geq 0$.

A $\xi$-$\mathcal{G}$projective resolution $\mathbf{G}\to M$ is said to be of length $n$ if $G_n\neq 0$ and $G_i=0$ for all $i> n$.

Assume $\xi$-$\mathcal{G}$pd$M=n<\infty$. By \cite[Proposition 5.5]{HZZ}, there is an $\mathbb{E}$-triangle $\xymatrix@C=0.5cm{K_1\ar[r]&G_0\ar[r]^{f_0}&M\ar@{-->}[r]&}$ in $\xi$ with $G_0\in\mathcal{GP}(\xi)$ and $\xi$-pd$K_1\leq n-1$. In particular,
$f_0$ is a $\xi$-$\mathcal{G}$projective precover of $M$. Inductively, we can get a $\xi$-$\mathcal{G}$projective resolution of length $n$ for $M$.

The notions of $\xi$-$\mathcal{G}$injective preenvelopes and $\xi$-$\mathcal{G}$injective coresolutions are given dually.

\begin{definition} {\rm Let $M,N\in\mathcal{C}$.
\begin{itemize}
  \item [(1)] Assume that $M$ admits a $\xi$-$\mathcal{G}$projective resolution $\mathbf{G}\to M$. For any integer $i\geq 0$, we define
$$
\xi{\rm xt}^i_{\mathcal{GP}(\xi)}(M,N)={\rm H}^i\mathcal{C}(\mathbf{G},N).
$$
  \item [(2)] Assume that $N$ admits a $\xi$-$\mathcal{G}$injective coresolution $N\to \mathbf{E}$. For any integer $i\geq 0$, we define
$$
\xi{\rm xt}^i_{\mathcal{GI}(\xi)}(M,N)={\rm H}^i\mathcal{C}(M,\mathbf{E}).
$$
\end{itemize}
}
\end{definition}

\begin{lem}
Let $M,M'\in\widetilde{\mathcal{GP}}(\xi)$. Consider $\xi$-projective resolutions $\pi:\mathbf{P}\to M$ and $\pi':\mathbf{P}'\to M'$, and
$\xi$-$\mathcal{G}$projective resolutions $\vartheta: \mathbf{G}\to M$ and $\vartheta': \mathbf{G}'\to M'$. Then
\begin{itemize}
  \item[(1)] there exist unique up to homotopy morphisms $\gamma: \mathbf{P}\to \mathbf{G}$ and $\gamma': \mathbf{P}'\to \mathbf{G}'$ such that $\pi=\vartheta\gamma$ and $\pi'=\vartheta'\gamma'$
  \item[(2)] for any morphism $\alpha: M\to M'$, there is a unique up to homotopy morphism $\tau: \mathbf{G}\to \mathbf{G}'$ such that the right square of the diagram
      \begin{equation}\label{homotopy}
      \begin{split}
        \xymatrix{\mathbf{P}\ar[r]^\gamma\ar[d]^{\tau'}&\mathbf{G}\ar[r]^\vartheta\ar[d]^{\tau}&M\ar[d]^{\alpha}\\
      \mathbf{P}'\ar[r]^{\gamma'}&\mathbf{G}'\ar[r]^{\vartheta'}&M'}
      \end{split}
      \end{equation}
      is commutative. Moreover, for each choice of $\tau$, there exists a unique up to homotopy morphism $\tau':\mathbf{P}\to \mathbf{P}'$ making the left square commute up to homotopy.
\end{itemize}
\end{lem}

\begin{proof}
Using standard arguments
from homological algebra, one can prove the corresponding version of the comparison
theorem for $\xi$-projective resolutions and  $\xi$-$\mathcal{G}$projective resolutions, that is, there are unique up to homotopy morphisms
$\tau':\mathbf{P}\to \mathbf{P}'$ and $\tau:\mathbf{G}\to \mathbf{G}'$ making the following diagrams
$$
\xymatrix{\mathbf{P}\ar[r]^{\pi}\ar[d]^{\tau'}&M\ar[d]^{\alpha}\\
\mathbf{P}'\ar[r]^{\pi'}&M'} \ \ \ \
\xymatrix{\mathbf{G}\ar[r]^{\vartheta}\ar[d]^{\tau}&M\ar[d]^{\alpha}\\
\mathbf{G}'\ar[r]^{\vartheta'}&M'}
$$
commute. Similarly, there are unique up to homotopy morphisms
$\gamma:\mathbf{P}\to \mathbf{G}$ and $\gamma':\mathbf{P}'\to \mathbf{G}'$ making the following diagrams
$$
\xymatrix{\mathbf{P}\ar[r]^{\pi}\ar[d]^{\gamma}&M\ar@{=}[d]\\
\mathbf{G}\ar[r]^{\vartheta}&M} \ \ \ \
\xymatrix{\mathbf{P}'\ar[r]^{\pi'}\ar[d]^{\gamma'}&M'\ar@{=}[d]\\
\mathbf{G}'\ar[r]^{\vartheta'}&M'}
$$
commute, i.e. (1) holds.

We next show that the left square of (\ref{homotopy}) is commutative up to homotopy. Firstly, we have a commutative diagram
$$
\xymatrix@C=1.5cm{\cdots\ar[r]&P_2\ar[r]^{d_2^{\mathbf{P}}}\ar[d]^{\tau_2\gamma_2-\gamma_2'\tau_2'}&
P_1\ar[r]^{d_1^{\mathbf{P}}}\ar[d]^{\tau_1\gamma_1-\gamma_1'\tau_1'}&P_0\ar[r]\ar[d]^{\tau_0\gamma_0-\gamma_0'\tau_0'}&0\\
\cdots\ar[r]&G_2'\ar[r]^{d_2^{\mathbf{G}'}}&
G_1'\ar[r]^{d_1^{\mathbf{G}'}}&G_0'\ar[r]&0.
}
$$
Note that for the $\xi$-$\mathcal{G}$projective resolution $\vartheta': \mathbf{G}'\to M'$, there are $\mathcal{C}(\mathcal{GP},-)$-exact $\mathbb{E}$-triangles
$\xymatrix@C=0.5cm{H_{i+1};\ar[r]^{u_i}&G_i'\ar[r]^{v_i}&H_i'\ar@{-->}[r]&
}$ such that $d_i^{\mathbf{G}'}=u_{i-1}v_i$ and $H_0'=M'$. Consider an exact sequence
$$
\xymatrix{0\ar[r]&\mathcal{C}(P_0,H_1')\ar[r]^{\mathcal{C}(P_0,u_0)}&\mathcal{C}(P_0,G_0')\ar[r]^{\mathcal{C}(P_0,v_0)}&\mathcal{C}(P_0,M')\ar[r]&0.
}
$$
Since $\mathcal{C}(P_0,v_0)(\tau_0\gamma_0-\gamma_0'\tau_0')=v_0(\tau_0\gamma_0-\gamma_0'\tau_0')=0$, there is $t_0\in\mathcal{C}(P_0,H_1')$ with
$u_0t_0=\mathcal{C}(P_0,u_0)(t_0)=\tau_0\gamma_0-\gamma_0'\tau_0'$. Moreover, by the exact sequence
$$
\xymatrix{0\ar[r]&\mathcal{C}(P_0,H_2')\ar[r]^{\mathcal{C}(P_0,u_1)}&\mathcal{C}(P_0,G_1')\ar[r]^{\mathcal{C}(P_0,v_1)}&\mathcal{C}(P_0,H_1')\ar[r]&0.
}
$$
there is $s_0\in\mathcal{C}(P_0,G_1')$ with $v_1s_0=\mathcal{C}(P_0,v_1)(s_0)=t_0$. Thus $\tau_0\gamma_0-\gamma_0'\tau_0'=u_0v_1s_0=d_1^{\mathbf{G}'}s_0$.

Consider an exact sequence
$$
\xymatrix{0\ar[r]&\mathcal{C}(P_1,H_2')\ar[r]^{\mathcal{C}(P_1,u_1)}&\mathcal{C}(P_1,G_1')\ar[r]^{\mathcal{C}(P_1,v_1)}&\mathcal{C}(P_1,H_1')\ar[r]&0.
}
$$
Let $r_1=\tau_1\gamma_1-\gamma_1'\tau_1'-s_0d_1^{\mathbf{P}}$. Then $\mathcal{C}(P_1,u_0)(\mathcal{C}(P_1,v_1)(r_1))=\mathcal{C}(P_1,d_1^{\mathbf{G}'})(r_1)=0$.
But $\mathcal{C}(P_1,u_0)$ is monic, we have $\mathcal{C}(P_1,v_1)(r_1)=0$. Thus there is $t_1\in\mathcal{C}(P_1,H_2')$ with $r_1=\mathcal{C}(P_1,u_1)(t_1)=u_1t_1$. By the exact sequence
$$
\xymatrix{0\ar[r]&\mathcal{C}(P_1,H_3')\ar[r]^{\mathcal{C}(P_1,u_2)}&\mathcal{C}(P_1,G_2')\ar[r]^{\mathcal{C}(P_1,v_2)}&\mathcal{C}(P_1,H_2')\ar[r]&0,
}
$$
there is $s_1\in\mathcal{C}(P_1,G_2')$ with $t_1=\mathcal{C}(P_1,v_2)(s_1)=v_2s_1$. Thus $r_1=u_1t_1=u_1v_2s_1=d_2^{\mathbf{G}'}s_1$, that is,
$\tau_1\gamma_1-\gamma_1'\tau_1'=d_2^{\mathbf{G}'}s_1+s_0d_1^{\mathbf{P}}$. Continuing this process, we obtain a homotopy $\{s_i\}$ such that $\tau\gamma\thicksim\gamma'\tau'$.
\end{proof}

\begin{rem}\label{fact:2.5'} {Let $M\in\widetilde{\mathcal{GP}}(\xi)$ and $N\in\widetilde{\mathcal{GI}}(\xi)$.  By the above lemma and its dual argument, one can see that ${\rm \xi xt}_{\mathcal{GP}(\xi)}^n(M,-)$ and ${\rm \xi xt}_{\mathcal{GI}(\xi)}^n(-,N)$ are  independent of the choice of $\xi$-$\mathcal{G}$projective resolutions of $M$ and $\xi$-$\mathcal{G}$injective coresolutions of $N$, respectively.
}
\end{rem}

Now we show the balance of $\xi$-Gorenstein cohomology.

\begin{thm}\label{thm:balance-of-gorenstein}
Assume that $M\in\widetilde{\mathcal{GP}}(\xi)$ and $N\in\widetilde{\mathcal{GI}}(\xi)$. Then
$$
{\rm \xi xt}_{\mathcal{GP}(\xi)}^n(M,N)\cong {\rm \xi xt}_{\mathcal{GI}(\xi)}^n(M,N)
$$
for any $n\geq 1$.
\end{thm}

\begin{proof}
Since $M\in\widetilde{\mathcal{GP}}(\xi)$, by \cite[Proposition 5.5]{HZZ}, there is an $\mathbb{E}$-triangle $\xymatrix@C=0.5cm{K_1\ar[r]^{g_1}&G_0\ar[r]^{f_0}&M\ar@{-->}[r]&}$ in $\xi$ with $G_0\in\mathcal{GP}(\xi)$ and $\xi$-pd$K_1<\infty$. For any $\xi$-$\mathcal{G}$injective object $H$, by definition
there is an $\mathbb{E}$-triangle $\xymatrix@C=0.5cm{H_1\ar[r]^s&E_0\ar[r]^t&H\ar@{-->}[r]&}$ in $\xi$ with $H_1\in\mathcal{GI}(\xi)$ and $E_0\in\mathcal{I}(\xi)$. Consider the following commutative diagram
$$
\xymatrix{&&0\ar[d]&0\ar[d]&\\
&&\mathcal{C}(M,E_0)\ar[r]\ar[d]&\mathcal{C}(M,H)\ar[d]&\\
&&\mathcal{C}(G_0,E_0)\ar[r]^{\mathcal{C}(G_0,t)}\ar[d]^{\mathcal{C}(g_1,E_0)}&\mathcal{C}(G_0,H)\ar[d]^{\mathcal{C}(g_1,H)}&\\
0\ar[r]& \mathcal{C}(K_1,H_1)\ar[r]& \mathcal{C}(K_1,E_0)\ar[r]^{\mathcal{C}(K_1,t)}\ar[d] & \mathcal{C}(K_1,H)\ar[r]\ar[d]& 0\\
&&0&\ 0.&}
$$
Since $\xi$-pd$K_1<\infty$ and $E_0\in\mathcal{I}(\xi)$, we have that the bottom row and the first column are exact. It follows that the second column is exact, and
hence $\xymatrix@C=0.5cm{K_1\ar[r]^{g_1}&G_0\ar[r]^{f_0}&M\ar@{-->}[r]&}$ is $\mathcal{C}(-,\mathcal{GI}(\xi))$-exact. Inductively, we get a $\xi$-$\mathcal{G}$projective resolution $\mathbf{G}\to M$ which is $\mathcal{C}(-,\mathcal{GI}(\xi))$-exact.

Dually, we can get a $\xi$-$\mathcal{G}$injective resolution $N\to \mathbf{E}$ which is $\mathcal{C}(\mathcal{GP}(\xi),-)$-exact. Following these, we have a commutative diagram as follows
$$
\xymatrix@C=20pt@R=20pt{&0\ar[d]&0\ar[d]&0\ar[d]&\\
0\ar[r]&\mathcal{C}(M,N)\ar[r]\ar[d]&\mathcal{C}(M,E^0)\ar[r]\ar[d]&\mathcal{C}(M,E^1)\ar[r]\ar[d]&\cdots\\
0\ar[r]&\mathcal{C}(G_0,N)\ar[r]\ar[d]&\mathcal{C}(G_0,E^0)\ar[r]\ar[d]&\mathcal{C}(G_0,E^1)\ar[r]\ar[d]&\cdots\\
0\ar[r]&\mathcal{C}(G_1,N)\ar[r]\ar[d]&\mathcal{C}(G_1,E^0)\ar[r]\ar[d]&\mathcal{C}(G_1,E^1)\ar[r]\ar[d]&\cdots\\
&\vdots&\vdots&\vdots&
}
$$
where all rows and columns are exact except the top row and the left column. By \cite[Proposition 1.4.16]{EJ2}, we have
$$
{\rm \xi xt}_{\mathcal{GP}(\xi)}^n(M,N)={\rm H}^n\mathcal{C}(\mathbf{G},N)\cong {\rm H}^n\mathcal{C}(M,\mathbf{E})= {\rm \xi xt}_{\mathcal{GI}(\xi)}^n(M,N),
$$
as desired.
\end{proof}

Next we compare ${\rm \xi xt}_{\mathcal{GP}(\xi)}^n(M,N)$ and ${\rm \xi xt}_{\mathcal{GI}(\xi)}^n(M,N)$ with ${\rm \xi xt}_{\xi}^n(M,N)$.

\begin{prop}
Let $M,N\in\mathcal{C}$.
\begin{itemize}
  \item [(1)] If $\xi$-${\rm pd}M<\infty$, then ${\rm \xi xt}_{\mathcal{GP}(\xi)}^n(M,N)\cong {\rm \xi xt}_{\xi}^n(M,N)$ for any $n\geq 0$.
  \item [(2)] If $\xi$-${\rm id}N<\infty$, then ${\rm \xi xt}_{\mathcal{GI}(\xi)}^n(M,N)\cong {\rm \xi xt}_{\xi}^n(M,N)$ for any $n\geq 0$.
\end{itemize}
\end{prop}

\begin{proof}
(1) Assume that $\xi$-${\rm pd}M=m<\infty$. Then there is a $\xi$-projective resolution
$$
0\to P_m\to \cdots \to P_1\to P_0\to M\to 0.
$$
For the relevant $\mathbb{E}$-triangle $\xymatrix@C=0.5cm{K_{i+1}\ar[r]&P_i\ar[r]&K_i\ar@{-->}[r]&}$, since all terms have finite $\xi$-projective dimension, it is $\mathcal{C}(\mathcal{GP}(\xi),-)$-exact by \cite[Lemma 3.5]{HZZ1}. This shows that the $\xi$-projective resolution above is a $\xi$-$\mathcal{G}$projective resolution. Thus ${\rm \xi xt}_{\mathcal{GP}(\xi)}^n(M,N)\cong {\rm \xi xt}_{\xi}^n(M,N)$.

(2) is dual.
\end{proof}

\begin{prop}\label{prop:long-exact1}
Let $M\in\widetilde{\mathcal{GP}}(\xi)$ and $\mathbf{N}:\ \xymatrix@C=0.5cm{N\ar[r]^x&N'\ar[r]^y&N''\ar@{-->}[r]&}$ a $\mathcal{C}(\mathcal{GP}(\xi),-)$-exact $\mathbb{E}$-triangle in $\xi$.
\begin{itemize}
  \item [(1)] There are the connecting maps $\varepsilon^i_{\mathcal{GP}}(M,\mathbf{N}):\xi{\rm xt}^{i}_{\mathcal{GP}(\xi)}(M,N'')\longrightarrow \xi{\rm xt}^{i+1}_{\mathcal{GP}(\xi)}(M,N)$ which are natural in $M$ and $\mathbf{N}$, such that the following sequence
\begin{multline*}
  0\longrightarrow \xi{\rm xt}^0_{\mathcal{GP}(\xi)}(M,N)\longrightarrow\xi{\rm xt}^0_{\mathcal{GP}(\xi)}(M,N')\longrightarrow\xi{\rm xt}^0_{\mathcal{GP}(\xi)}(M,N'')\longrightarrow \xi{\rm xt}^1_{\mathcal{GP}(\xi)}(M,N)\longrightarrow\\
  \cdots\longrightarrow \xi{\rm xt}^{n-1}_{\mathcal{GP}(\xi)}(M,N'')\longrightarrow \xi{\rm xt}^{n}_{\mathcal{GP}(\xi)}(M,N)\longrightarrow\xi{\rm xt}^n_{\mathcal{GP}(\xi)}(M,N')\longrightarrow\xi{\rm xt}^n_{\mathcal{GP}(\xi)}(M,N'')\longrightarrow \cdots
\end{multline*}
is exact
  \item [(2)] There are maps $\delta^{i}(M,N''): \xi{\rm xt}^{i}_{\mathcal{GP}(\xi)}(M,N'')\rightarrow \xi{\rm xt}^{i}_{\xi}(M,N'')$ and \\ $\delta^{i}(M,N): \xi{\rm xt}^{i}_{\mathcal{GP}(\xi)}(M,N)\rightarrow \xi{\rm xt}^{i}_{\xi}(M,N)$ such that the following diagram
$$
\xymatrix{\xi{\rm xt}^{i}_{\mathcal{GP}(\xi)}(M,N'')\ar[r]\ar[d]^{\delta^{i}(M,N'')}&\xi{\rm xt}^{i+1}_{\mathcal{GP}(\xi)}(M,N)\ar[d]^{\delta^{i}(M,N)}\\
\xi{\rm xt}^{i}_{\xi}(M,N'')\ar[r]&\xi{\rm xt}^{i+1}_{\xi}(M,N)}
$$
is commutative for each $i\geq 0$.
\end{itemize}
\end{prop}

\begin{proof}
Let $\pi:\mathbf{P}\to M$ and $\xi:\mathbf{G}\to M$ be $\xi$-projective and $\xi$-$\mathcal{G}$projective resolutions, respectively. Then there is a morphism
$\gamma: \mathbf{P}\to \mathbf{G}$ which induced a commutative diagram
$$
\xymatrix{0\ar[r]&\mathcal{C}(\mathbf{G},N)\ar[r]\ar[d]^{\mathcal{C}(\gamma,N)}&\mathcal{C}(\mathbf{G},N')\ar[r]\ar[d]^{\mathcal{C}(\gamma,N')}&
\mathcal{C}(\mathbf{G},N'')\ar[r]\ar[d]^{\mathcal{C}(\gamma,N'')}&0\\
0\ar[r]&\mathcal{C}(\mathbf{P},N)\ar[r]&\mathcal{C}(\mathbf{P},N')\ar[r]&\mathcal{C}(\mathbf{P},N'')\ar[r]&0.}
$$
Here the two rows are short exact sequences of complexes. By taking the homology group, we get the desired long exact sequence and the commutative diagram.
\end{proof}

Using standard arguments
from relative homological algebra, one can prove the following version of the Horseshoe Lemma for $\xi$-$\mathcal{G}$projective resolutions.

\begin{lem} {\rm (Horseshoe Lemma for $\xi$-$\mathcal{G}$projective resolutions)}
Let $\xymatrix{M\ar[r]^x& M'\ar[r]^y& M'' \ar@{-->}[r]^{\delta}& }$ be a $\mathcal{C}(\mathcal{GP}(\xi),-)$-exact $\mathbb{E}$-triangle in $\xi$ such that $\xi\mbox{-}\mathcal{G}{\rm pd}M<\infty$ and $\xi\mbox{-}\mathcal{G}{\rm pd}M''<\infty$. Let $\pi:\mathbf{P}\to M$ and $\pi'':\mathbf{P}''\to M''$ be $\xi$-projective resolutions of $M$ and $M''$, respectively.
Let $\vartheta:\mathbf{G}\to M$ and $\vartheta'':\mathbf{G}''\to M''$ be $\xi$-$\mathcal{G}$projective resolutions of $M$ and $M''$, respectively.  Then there is a commutative diagram:
\begin{equation}\label{com}
\begin{split}
\xymatrix{\mathbf{P}\ar[r]\ar[d]^{\gamma}&\mathbf{P}'\ar[r]\ar[d]^{\gamma'}&\mathbf{P}''\ar[d]^{\gamma''}\\
\mathbf{G}\ar[r]\ar[d]^{\vartheta}&\mathbf{G}'\ar[r]\ar[d]^{\vartheta'}&\mathbf{G}''\ar[d]^{\vartheta''}\\
M\ar[r]^x& M'\ar[r]^y& M''}
\end{split}
\end{equation}
such that $\pi=\vartheta\gamma$, $\pi''=\vartheta''\gamma''$, $\vartheta':\mathbf{G}'\to M'$ is a $\xi$-$\mathcal{G}$projective resolution of $M'$ and
 $\pi'=\vartheta'\gamma':\mathbf{P}'\to M'$ is a $\xi$-projective resolution of $M'$. Moreover,  the two upper rows are split $\mathbb{E}$-triangle in $\xi$.
\end{lem}

\begin{prop}\label{prop:long-exact2}
Let $N$ be an object in $\mathcal{C}$ and ${\mathbf{M}}:\xymatrix{M\ar[r]^x& M'\ar[r]^y& M'' \ar@{-->}[r]^{\delta}& }$  a $\mathcal{C}(\mathcal{GP}(\xi),-)$-exact $\mathbb{E}$-triangle in $\xi$  such that $\xi\mbox{-}\mathcal{G}{\rm pd}M<\infty$ and $\xi\mbox{-}\mathcal{G}{\rm pd}M''<\infty$.

\begin{itemize}
  \item [(1)] There are homomorphisms $\varepsilon^i_{\mathcal{GP}}(\mathbf{M},N):\xi{\rm xt}^{i}_{\mathcal{GP}(\xi)}(M,N)\to \xi{\rm xt}^{i+1}_{\mathcal{GP}(\xi)}(M'',N)$ natural in $\mathbf{M}$ and $N$ such that the following sequence
\begin{multline*}
  0\longrightarrow \xi{\rm xt}^0_{\mathcal{GP}(\xi)}(M'',N)\longrightarrow\xi{\rm xt}^0_{\mathcal{GP}(\xi)}(M',N)\longrightarrow\xi{\rm xt}^0_{\mathcal{GP}(\xi)}(M,N)\longrightarrow \xi{\rm xt}^1_{\mathcal{GP}(\xi)}(M'',N)\longrightarrow\\
  \cdots\longrightarrow \xi{\rm xt}^{n-1}_{\mathcal{GP}(\xi)}(M,N)\longrightarrow \xi{\rm xt}^{n}_{\mathcal{GP}(\xi)}(M'',N)\longrightarrow\xi{\rm xt}^n_{\mathcal{GP}(\xi)}(M',N)\longrightarrow\xi{\rm xt}^n_{\mathcal{GP}(\xi)}(M,N)\longrightarrow \cdots
\end{multline*}
is exact

  \item [(2)] There are maps $\delta^{i}(M,N): \xi{\rm xt}^{i}_{\mathcal{GP}(\xi)}(M,N'')\rightarrow \xi{\rm xt}^{i}_{\xi}(M,N)$ and \\ $\delta^{i}(M,N): \xi{\rm xt}_{\mathcal{GP}(\xi)}^{i}(M'',N)\rightarrow \xi{\rm xt}^{i}_{\xi}(M'',N)$ such that the following diagram
$$
\xymatrix{\xi{\rm xt}^{i}_{\mathcal{GP}(\xi)}(M,N)\ar[r]^{\varepsilon^i_{\mathcal{GP}}(\mathbf{M},N)}\ar[d]^{\delta^{i}(M,N)}&\xi{\rm xt}^{i+1}_{\mathcal{GP}(\xi)}(M'',N)\ar[d]^{\delta^{i+1}(M'',N)}\\
\xi{\rm xt}^{i}_{\xi}(M,N)\ar[r]^{\varepsilon^i_{\mathcal{P}}(\mathbf{M},N)}&\xi{\rm xt}^{i+1}_{\xi}(M'',N)}
$$
is commutative for each $i\geq 0$.
\end{itemize}
\end{prop}

\begin{proof}
Since $A$ and $C$ have finite $\xi$-$\mathcal{G}$projective dimensions, we can construct the diagram (\ref{com}).
Moreover, since the two upper rows of (\ref{com}) are split $\mathbb{E}$-triangles in $\xi$, by applying the functor $\mathcal{C}(-,N)$ we can get
a commutative diagram  of complexes
$$
\xymatrix{0\ar[r]&\mathcal{C}(\mathbf{G}'',N)\ar[r]\ar[d]&\mathcal{C}(\mathbf{G}',N)\ar[r]\ar[d]&\mathcal{C}(\mathbf{G},N)\ar[r]\ar[d]&0\\
0\ar[r]&\mathcal{C}(\mathbf{P}'',N)\ar[r]&\mathcal{C}(\mathbf{P}',N)\ar[r]&\mathcal{C}(\mathbf{P},N)\ar[r]&0}
$$
with exact rows.
By taking the homology group, we get the desired long exact sequence and the commutative diagram.
\end{proof}

\section{\bf The Avramov-Martsinkovsky type exact sequence}

In \cite{HZZZ}, we introduced the notion of $\xi$-complete cohomology in an  extriangulated category. In this section, we will give
an Avramov-Martsinkovsky type exact sequence which connects $\xi$-cohomology, $\xi$-Gorenstein cohomology and $\xi$-complete cohomology. In particular,
we can use $\xi$-complete cohomology to measure the distance  between $\xi$-cohomology and $\xi$-Gorenstein cohomology.

We denote by $\textrm{Ch}(\mathcal{C})$ the category of complexes in $\mathcal{C}$; the objects are complexes and morphisms are chain maps. We write the complexes homologically, so an object $\mathbf{X}$ of $\textrm{Ch}(\mathcal{C})$ is of the form
$$\xymatrix@C=2em{\mathbf{X}:=\cdots \ar[r]&X_{n+1}\ar[r]^{d_{n+1}^{\mathbf{X}}}&X_n\ar[r]^{d_n^{\mathbf{X}}}&X_{n-1}\ar[r]&\cdots}.$$
The \emph{$i$th shift} of $\mathbf{X}$ is the complex $\mathbf{X}[i]$ with $n$th component $\mathbf{X}_{n-i}$ and differential $d_n^{\mathbf{X}[i]}=(-1)^{i}d_{n-i}^{\mathbf{X}}$. Assume that $\mathbf{X}$ and $\mathbf{Y}$ are complexes in $\textrm{Ch}(\mathcal{C})$.
A homomorphism $\xymatrix@C=2em{\varphi:\mathbf{X}\ar[r]&\mathbf{Y}}$ of degree $n$ is a family $(\varphi_i)_{i\in\mathbb{Z}}$ of morphisms $\xymatrix@C=2em{\varphi_i:X_i\ar[r]& Y_{i+n}}$ for all $i\in\mathbb{Z}$. In this case, we set $|\varphi|=n$. All such homomorphisms form an abelian group, denoted by $\mathcal{C}(\mathbf{X},\mathbf{Y})_n$, which is identified with $\prod_{i\in \mathbb{Z}}{\rm \mathcal{C}}(X_i,Y_{i+n})$. We let $\mathcal{C}(\mathbf{X},\mathbf{Y})$ be the complex of abelian groups with $n$th component $\mathcal{C}(\mathbf{X},\mathbf{Y})_n$ and differential $d(\varphi_i)=d_{i+n}^{\mathbf{Y}}\varphi_i-(-1)^n\varphi_{i-1}d_i^{\mathbf{X}}$ for $\varphi=(\varphi_i)\in\mathcal{C}(\mathbf{X},\mathbf{Y})_n$.
We refer to \cite{AFH,CFH} for more details.

Let $M$ and $N$ be objects in $\mathcal{C}$.
\begin{enumerate}
\item There are two $\xi$-projective resolutions $\xymatrix@C=2em{\mathbf{P}_M\ar[r]& M}$ and $\xymatrix@C=2em{\mathbf{P}_N\ar[r]& N}$ of $M$ and $N$, respectively. A homomorphism $\beta\in \mathcal{C}(\mathbf{P}_M,\mathbf{P}_N)$ is {\it bounded above} if $\beta_i=0$ for all $i\gg 0$. The subset $\overline{\mathcal{C}}(\mathbf{P}_M,\mathbf{P}_N)$, consisting of all bounded above homomorphisms, is a subcomplex with components
\begin{center}$\overline{\mathcal{C}}(\mathbf{P}_M,\mathbf{P}_N)_n=\{(\varphi_i)\in \mathcal{C}(\mathbf{P}_M,\mathbf{P}_N)_n \ | \ \varphi_i=0$ for all $i\gg 0\}.$\end{center}
We set
\begin{equation}\label{1}
  \widetilde{\mathcal{C}}(\mathbf{P}_M,\mathbf{P}_N)={\mathcal{C}}(\mathbf{P}_M,\mathbf{P}_N)/\overline{\mathcal{C}}(\mathbf{P}_M,\mathbf{P}_N).
\end{equation}

\item There are two $\xi$-injective coresolutions $\xymatrix@C=2em{M\ar[r]&\mathbf{I}_M}$ and $\xymatrix{N\ar[r]&\mathbf{I}_N}$ of $M$ and $N$, respectively. A homomorphism $\beta\in {\mathcal{C}}(\mathbf{I}_M,\mathbf{I}_N)$ is {\it bounded below} if $\beta_i=0$ for all $i\ll 0$. The subset $\underline{\mathcal{C}}(\mathbf{I}_M,\mathbf{I}_N)$, consisting of all bounded below homomorphisms, is a subcomplex with components
\begin{center}$\underline{\mathcal{C}}(\mathbf{I}_M,\mathbf{I}_N)_n=\{(\varphi_i)\in {\mathcal{C}}(\mathbf{I}_M,\mathbf{I}_N)_n \ | \ \varphi_i=0$ for all $i\ll 0\}.$\end{center}
We set
\begin{equation}\label{2}
  \widetilde{\mathcal{C}}(\mathbf{I}_M,\mathbf{I}_N)={\mathcal{C}}(\mathbf{I}_M,\mathbf{I}_N)/\underline{\mathcal{C}}(\mathbf{I}_M,\mathbf{I}_N).
\end{equation}
\end{enumerate}

\begin{definition}\label{df:3.6} {\rm (see \cite[Definition 3.4]{HZZZ})} {\rm Let $M$ and $N$ be objects in $\mathcal{C}$, and let $n$ be an integer.
\begin{enumerate}
\item Using $\xi$-projective resolutions, we define the $n$th \emph{$\xi$-complete cohomology group}, denoted by $\widetilde{\rm \xi xt}_{\mathcal{P}}^n(M,N)$, as
$$\widetilde{\rm \xi xt}_{\mathcal{P}}^n(M,N)=H^n(\widetilde{\mathcal{C}}(\mathbf{P}_M,\mathbf{P}_N)),$$
where $\widetilde{\mathcal{C}}(\mathbf{P}_M,\mathbf{P}_N)$ is the complex (\ref{1}).

\item Using $\xi$-injective coresolutions, we define the $n$th \emph{$\xi$-complete cohomology group}, denoted by $\widetilde{\rm \xi xt}_{\mathcal{I}}^n(M,N)$, as
$$\widetilde{\rm \xi xt}_{\mathcal{I}}^n(M,N)=H^n(\widetilde{\mathcal{C}}(\mathbf{I}_M,\mathbf{I}_N)),$$
where $\widetilde{\mathcal{C}}(\mathbf{I}_M,\mathbf{I}_N)$ is the complex (\ref{2}).
\end{enumerate}}
\end{definition}

\begin{definition}\label{df:3.2} {\rm (see \cite[Definition 4.3]{HZZZ})} {\rm Let $M\in\mathcal{C}$ be an object. A \emph{$\xi$-complete resolution} of $M$ is a diagram $$\xymatrix@C=2em{\mathbf{T}\ar[r]^{\nu}&\mathbf{P}\ar[r]^{\pi}&M}$$ of morphisms of complexes satisfying:
  (1)  $\pi:\mathbf{P}\ra M$ is a $\xi$-projective resolution of $M$;
  (2) $\mathbf{T}$ is a complete $\xi$-projective resolution;
  (3) $\nu:\mathbf{T}\ra \mathbf{P}$ is a morphism such that $\nu_{i}$ $=$ {\rm id$_{T_{i}}$} for all $i\gg 0$.
   Moreover, a $\xi$-complete resolution is \emph{split} if $\nu_{i}$ has a section {(}i.e., there exists a morphism $\eta_{i}:{P}_{i}\rightarrow {T}_{i}$ such that $\nu_{i}\eta_{i}={\rm id}_{{P}_{i}}${)} for all $i\in{\mathbb{Z}}$.
 }
\end{definition}

The following lemma is very key, which shows that one can compute $\xi$-complete cohomology for objects having finite $\xi$-$\mathcal{G}$projective dimension using $\xi$-complete resolutions.

\begin{lem}\label{lem:complete-cohomology} {\rm (see \cite[Theorem 4.6]{HZZZ})} Let $M$ and $N$ be objects in $\mathcal{C}$. If $M$ admits a $\xi$-complete resolution $\xymatrix@C=2em{\mathbf{T}\ar[r]^{\nu}&\mathbf{P}\ar[r]^{\pi}&M,}$ then for any integer $i$, there exists  an isomorphism
$$\widetilde{\xi{\rm xt}}_{\mathcal{P}}^i(M,N)\cong H^i(\mathcal{C}(\mathbf{T},N)).$$
\end{lem}

Assume that
 $M$ has a $\xi$-complete resolution $\xymatrix@C=2em{\mathbf{T}\ar[r]^{\nu}&\mathbf{P}\ar[r]^{\pi}&M}$ such that $\nu_{i}$ is an isomorphism for each $i\geq n$.
 By \cite[Proposition 4.4]{HZZZ},  there is a split $\xi$-complete resolution $\xymatrix@C=2em{\mathbf{S}\ar[r]^{\mu}&\mathbf{P}\ar[r]^{\pi}&M}$ such that $\mu_{i}$ is an isomorphism for each $i\geq n$. Now we need a new  construction as follows, which seems to be similar to that of \cite[Proposition 4.4 (2) $\Rightarrow$ (3)]{HZZZ} but different.
 By assumption, there is a commutative
diagram
$$\xymatrix@C=2em{\mathbf{T}:=&\cdots\ar[r]&P_{n}\ar@{=}[d]\ar[r]&T_{n-1}\ar[r]\ar[d]^{\nu_{n-1}}&\cdots\ar[r]&T_{1}\ar[r]\ar[d]^{\nu_{1}}&
T_{0}\ar[d]^{\nu_{0}}\ar[r]&T_{-1}\ar[d]^{\nu_{-1}}\ar[r]&\cdots\\
\mathbf{P}:=&\cdots\ar[r]&P_{n}\ar[r]&P_{n-1}\ar[r]&\cdots\ar[r]&P_{1}\ar[r]&
P_{0}\ar[r]&0\ar[r]&\cdots}$$
with the $\mathbb{E}$-triangles $\xymatrix{K_{i+1}\ar[r]^{f_{i}}&P_{i}\ar[r]^{g_i}&K_{i}\ar@{-->}[r]&}$ and $\xymatrix{K'_{i+1}\ar[r]^{f'_{i}}&T_{i}\ar[r]^{g'_i}&K'_{i}\ar@{-->}[r]&}$ (Here $K'_n=K_n$) in $\xi$.
   Then we have the following morphism  of $\mathbb{E}$-triangles in $\xi$
$$\xymatrix{K_n\ar@{=}[d]\ar[r]^{g_{n-1}'}&T_{n-1}\ar[r]^{f_{n-1}'}\ar[d]^{\nu_{n-1}}&K_{n-1}'\ar[d]^{\omega_{n-1}}\ar@{-->}[r]^{\rho_{n-1}}&\\
K_n\ar[r]^{g_{n-1}}&P_{n-1}\ar[r]^{f_{n-1}}&K_{n-1}\ar@{-->}[r]^{\delta_{n-1}}&.}$$
Moreover, for any integer $i<n$, we have the following morphism of $\mathbb{E}$-triangles in $\xi$
$$\xymatrix{K_{i}'\ar[d]^{\omega_{i}}\ar[r]^{g_{i-1}'}&T_{i-1}\ar[r]^{f_{n-1}'}\ar[d]^{\nu_{i-1}}&K_{i-1}'\ar[d]^{\omega_{i-1}}\ar@{-->}[r]^{\rho_{i-1}}&\\
K_i\ar[r]^{g_{i-1}}&P_{i-1}\ar[r]^{f_{i-1}}&K_{i-1}\ar@{-->}[r]^{\delta_{i-1}}&.}$$
By \cite[Lemma 4.1]{HZZZ}, there is an $\mathbb{E}$-triangle in $\xi$
$$\xymatrix{K_n\ar[r]^{\tiny\begin{bmatrix}-g_{n-1}\\g_{n-1}'\end{bmatrix}\ \ \ \ \ \ \ }&P_{n-1}\oplus T_{n-1}\ar[r]^{\tiny \begin{bmatrix}1&\nu_{n-1}\\0&f_{n-1}'\end{bmatrix}\ \ \ }&P_{n-1}\oplus K_{n-1}'\ar@{-->}[r]^{\ \ \ \ \ \ \ \ \ \ \ \tiny\begin{bmatrix}0&1\end{bmatrix}^*\rho_{n-1}}&.}$$
 Since the morphism $\tiny\begin{bmatrix}1&0\end{bmatrix}:P_{n-1}\oplus K_{n-1}'\rightarrow P_{n-1}$ is a split epimorphism, and  it is a $\xi$-deflation. Hence ${\tiny\begin{bmatrix}1&\nu_{n-1}\end{bmatrix}=\begin{bmatrix}1&0\end{bmatrix}\begin{bmatrix}1&\nu_{n-1}\\0&f_{n-1}'\end{bmatrix}}: P_{n-1}\oplus T_{n-1}\rightarrow P_{n-1}$ is a $\xi$-deflation by \cite[Corollary 3.5]{HZZ}. Let $\xymatrix{L_{n-1}\ar[r]&P_{n-1}\oplus T_{n-1}\ar[r]^-{{\tiny\begin{bmatrix}1&\nu_{n-1}\end{bmatrix}}}&P_{n-1}\ar@{-->}[r]&}$ be an $\mathbb{E}$-triangle in $\xi$. Moreover, by  \cite[Lemma 3.7(2)]{HZZ} one has an $\mathbb{E}$-triangle
 $\xymatrix{K_n\ar[r]^{\tiny\begin{bmatrix}0\\g_{n-1}'\end{bmatrix}\ \ \ \ \ \ \ }&P_{n-1}\oplus T_{n-1}\ar[r]^{\tiny \begin{bmatrix}1&0\\0&f_{n-1}'\end{bmatrix}\ \ \ }&P_{n-1}\oplus K_{n-1}'\ar@{-->}[r]^{\ \ \ \ \ \ \ \ \ \ \ \tiny\begin{bmatrix}0&1\end{bmatrix}^*\rho_{n-1}}&}$ in $\xi$. By \cite[Lemma 5.9]{NP}, there is a commutative diagram
 $$
 \xymatrix{
 0\ar[r]\ar[d]&L_{n-1}\ar[r]\ar[d]&K_{n-1}''\ar[d]\ar@{-->}[r]&\\
 K_n\ar@{=}[d]\ar[r]^{\tiny\begin{bmatrix}0\\g_{n-1}'\end{bmatrix}\ \ \ \ \ \ \ }&P_{n-1}\oplus T_{n-1}\ar[d]^{\tiny\begin{bmatrix}1&\nu_{n-1}\end{bmatrix}}\ar[r]^{\tiny \begin{bmatrix}1&0\\
 0&f_{n-1}'\end{bmatrix}\ \ \ }&P_{n-1}\oplus K_{n-1}'\ar[d]_{\eta_{n-1}}\ar@{-->}[r]^{\ \ \ \ \ \ \ \ \ \ \ \tiny\begin{bmatrix}0&1\end{bmatrix}^*\rho_{n-1}}&\\
K_n\ar[r]^{g_{n-1}}\ar@{-->}[d]&P_{n-1}\ar@{-->}[d]\ar[r]^{f_{n-1}}&K_{n-1}\ar@{-->}[d]\ar@{-->}[r]^{\delta_{n-1}}&\\
&&&
 }
 $$
in which all rows and columns are $\mathbb{E}$-triangles.
Dual to \cite[Lemma 3.7(2)]{HZZ}, there exists an $\mathbb{E}$-triangle
$$\xymatrix{P_{n-1}\oplus K_{n-1}'\ar[r]^{\tiny\begin{bmatrix}1&0\\0&g_{n-2}'\end{bmatrix}\ \ \ }&P_{n-1}\oplus T_{n-2}\ar[r]^{\tiny \ \ \ \ \begin{bmatrix}0&f_{n-2}'\end{bmatrix}}& K_{n-2}'\ar@{-->}[r]^{\ \  \tiny\begin{bmatrix}0\\1\end{bmatrix}_*\rho_{n-2}}&,}$$
which is also in $\xi$ because $\xi$ is closed under cobase change. Since $K_{n-2}'\in\mathcal{GP}(\xi)$, by \cite[Lemma 4.10(2)]{HZZ}  we have the following morphism of $\mathbb{E}$-triangles in $\xi$
$$\xymatrix@C=4em{P_{n-1}\oplus K_{n-1}'\ar[d]_{\eta_{n-1}}\ar[r]^{\tiny\begin{bmatrix}1&0\\0&g_{n-2}'\end{bmatrix}}
&P_{n-1}\oplus T_{n-2}\ar[r]^{\ \ \tiny\begin{bmatrix}0&f_{n-2}'\end{bmatrix}}\ar[d]^{\gamma_{n-2}=[\gamma_{n-2}' \ \gamma_{n-2}']}
&K_{n-2}'\ar[d]^{\omega_{n-2}}\ar@{-->}[r]^{\ \ \tiny\begin{bmatrix}0\\1\end{bmatrix}_*\rho_{n-2}}&\\
K_{n-1}\ar[r]^{g_{n-2}}&P_{n-2}\ar[r]^{f_{n-2}}&K_{n-2}\ar@{-->}[r]^{\delta_{n-2}}&.}$$
Set $g_{n-2}\eta_{n-1}=[\alpha' \ \alpha'']$. By \cite[Lemma 4.1]{HZZZ}, there is an $\mathbb{E}$-triangle
$$\xymatrix@C=4em{P_{n-1}\oplus K_{n-1}'\ar[r]^{\tiny\begin{bmatrix}\alpha'&\alpha''\\1&0\\0&g_{n-2}'\end{bmatrix}\ \ \ \ \ \ }
&P_{n-2}\oplus P_{n-1}\oplus T_{n-2}\ar[r]^{\ \ \tiny\begin{bmatrix}1&\gamma_{n-2}'&\gamma_{n-2}'\\0&0&f_{n-2}'\end{bmatrix}}
&P_{n-2}\oplus K_{n-2}'\ar@{-->}[r]^{\ \ \ \ \ \tiny\begin{bmatrix}0&1\end{bmatrix}^*\begin{bmatrix}0\\1\end{bmatrix}_*\rho_{n-2}}&.}$$
Then $[1\ \ \gamma_{n-2}]=[1\ \ 0]{\  \tiny\begin{bmatrix}1&\gamma_{n-2}'&\gamma_{n-2}'\\0&0&f_{n-2}'\end{bmatrix}}$ is a $\xi$-deflation, and thus
there is an $\mathbb{E}$-triangle $$\xymatrix{L_{n-2}\ar[r]&P_{n-2}\oplus P_{n-1}\oplus T_{n-2}\ar[r]^-{[1\ \ \gamma_{n-2}]}&P_{n-2}\ar@{-->}[r]&}$$ in $\xi$.
By \cite[Lemma 5.9]{NP}, we have the following commutative diagram
$$\xymatrix@C=4em{K_{n-1}''\ar[r]\ar[d]&L_{n-2}\ar[r]\ar[d]&K_{n-2}''\ar@{-->}[r]\ar[d]&\\
P_{n-1}\oplus K_{n-1}'\ar[d]_{\eta_{n-1}}\ar[r]^{\tiny\begin{bmatrix}0&0\\1&0\\0&g_{n-2}'\end{bmatrix}\ \ \ \ \ \ }
&P_{n-2}\oplus P_{n-1}\oplus T_{n-2}\ar[r]^{\ \ \tiny\begin{bmatrix}1&0&0\\0&0&f_{n-2}'\end{bmatrix}}\ar[d]^{[1\ \ \gamma_{n-2}]}
&P_{n-2}\oplus K_{n-2}'\ar[d]^{\eta_{n-2}}\ar@{-->}[r]^{\ \ \ \ \ \tiny\begin{bmatrix}0&1\end{bmatrix}^*\begin{bmatrix}0\\1\end{bmatrix}_*\rho_{n-2}}&\\
K_{n-1}\ar[r]^{g_{n-2}}\ar@{-->}[d]&P_{n-2}\ar[r]^{f_{n-2}}\ar@{-->}[d]&K_{n-2}\ar@{-->}[r]^{\delta_{n-2}}\ar@{-->}[d]&\\
&&&
}$$
in which all rows and columns are $\mathbb{E}$-triangles.
By proceeding in this manner, we set
$$S_i=\left\{
\begin{array}{cc}
P_i& i\geq n\\
P_{n-1}\oplus T_{n-1}& i=n-1\\
P_i\oplus P_{i+1}\oplus T_i& i<n-1
\end{array}\right.
$$
$$\mu_i=\left\{
\begin{array}{cc}
1& i\geq n\\
{\tiny\begin{bmatrix}1&\nu_{n-1}\end{bmatrix}}& i=n-1\\
{\tiny\begin{bmatrix}1&\gamma_i'&\gamma_i''\end{bmatrix}}&0\leq i<n-1\\
0&i<0
\end{array}\right.
$$
Consequently, we get a commutative diagram
\begin{equation}\label{diag1}
\begin{split}
  \xymatrix@C=2em{
\mathbf{L}\ar[d]^\varsigma&\cdots\ar[r]&0\ar[r]\ar[d]&L_{n-1}\ar[r]\ar[d]&\cdots\ar[r]&L_1\ar[r]\ar[d]&L_0\ar[r]\ar[d]&S_{-1}\ar[r]\ar@{=}[d]&\cdots\\
\mathbf{S}\ar[d]^\mu&\cdots\ar[r]&P_{n}\ar@{=}[d]\ar[r]&S_{n-1}\ar[r]\ar[d]^{\mu_{n-1}}&\cdots\ar[r]&S_{1}\ar[r]\ar[d]^{\mu_{1}}&
S_{0}\ar[d]^{\mu_{0}}\ar[r]&S_{-1}\ar[d]^{\mu_{-1}}\ar[r]&\cdots\\
\mathbf{P}&\cdots\ar[r]&P_{n}\ar[r]&P_{n-1}\ar[r]&\cdots\ar[r]&P_{1}\ar[r]&
P_{0}\ar[r]&0\ar[r]&\cdots.}
\end{split}
\end{equation}
Note that every $S_i$ is $\xi$-projective, and $\mathbf{S}$ is obtained by pasting together those $\mathbb{E}$-triangles $$\xymatrix{K_n\ar[r]&P_{n-1}\oplus T_{n-1}\ar[r]&P_{n-1}\oplus K_{n-1}'\ar@{-->}[r]&}$$ and $$\xymatrix{P_i\oplus K_i'\ar[r]&P_{i-1}\oplus P_i\oplus T_{i-1}\ar[r]&P_{i-1}\oplus K_{i-1}'\ar@{-->}[r]&}$$ for all $i<n$, then the complex $\mathbf{S}$ is $\xi$-exact and $\mathcal{C}(-,\mathcal{P}(\xi))$-exact.
Moreover, since all columns are split $\mathbb{E}$-triangles, we can get the top row  is $\mathcal{C}(-,\mathcal{P}(\xi))$-exact. In particular, $\mathbf{L}$ is
a $\xi$-exact complex.

Now we  give
an Avramov-Martsinkovsky type exact sequence in extriangulated  category as follows.

\begin{thm}\label{AM}
Assume that
 $M$ admits a $\xi$-complete resolution $\xymatrix@C=2em{\mathbf{T}\ar[r]^{\nu}&\mathbf{P}\ar[r]^{\pi}&M}$. Then there are homomorphisms natural in $M$ and $N$,
 such that the following sequence
 \begin{multline*}
   \xymatrix@C=0.5cm{0\ar[r]&K\ar[r]&{\xi}{\rm xt}^1_{\mathcal{GP}(\xi)}(M,N)\ar[r]&{\xi}{\rm xt}^1_{\xi}(M,N)\ar[r]&\widetilde{{\xi}{\rm xt}}^1_{\mathcal{P}}(M,N)\ar[r]&{\xi}{\rm xt}^2_{\mathcal{GP}(\xi)}(M,N)\ar[r]&\cdots} \\
   \xymatrix@C=0.5cm{\cdots\ar[r]&\widetilde{{\xi}{\rm xt}}^{i-1}_{\mathcal{P}}(M,N)\ar[r]&{\xi}{\rm xt}^i_{\mathcal{GP}(\xi)}(M,N)\ar[r]&{\xi}{\rm xt}^i_{\xi}(M,N)\ar[r]&\widetilde{{\xi}{\rm xt}}^i_{\mathcal{P}}(M,N)\ar[r]&\cdots}
 \end{multline*}
 is exact.
\end{thm}

\begin{proof}
Assume that
 $M$ has a $\xi$-complete resolution $\xymatrix@C=2em{\mathbf{T}\ar[r]^{\nu}&\mathbf{P}\ar[r]^{\pi}&M}$ such that $\nu_{i}$ is an isomorphism for each $i\geq n$. By the previous argument, we have the diagram (\ref{diag1}). In particular, we have a commutative diagram
\begin{equation}\label{diag2}
\begin{split}
  \xymatrix@C=2em{
\mathbf{G}[-1]:\ar[d]&\cdots\ar[r]&0\ar[r]\ar[d]&L_{n-1}\ar[r]\ar[d]&\cdots\ar[r]&L_1\ar[r]\ar[d]&L_0\ar[r]\ar[d]&K_0'\oplus P_0\ar[r]\ar@{=}[d]&0\\
\mathbf{S}_{\succeq -1}:\ar[d] &\cdots\ar[r]&P_{n}\ar@{=}[d]\ar[r]&S_{n-1}\ar[r]\ar[d]^{\mu_{n-1}}&\cdots\ar[r]&S_{1}\ar[r]\ar[d]^{\mu_{1}}&
S_{0}\ar[d]^{\mu_{0}}\ar[r]&K_0'\oplus P_0\ar[d]\ar[r]&0\\
\mathbf{P}:&\cdots\ar[r]&P_{n}\ar[r]&P_{n-1}\ar[r]&\cdots\ar[r]&P_{1}\ar[r]&
P_{0}\ar[r]&0\ar[r]&0.}
\end{split}
\end{equation}
  Set $G_i=L_{i-1}$ for each $1\leq i\leq n$ and $G_0=K_0'\oplus P_0$. Then each $G_i$ is $\xi$-projective for $1\leq i\leq n$ and $G_0$ is $\xi$-$\mathcal{G}$projective. In the relevant $\mathbb{E}$-triangle $\xymatrix{K_{i+1}''\ar[r]&L_i\ar[r]&K_i''\ar@{-->}[r]&}$ for each $i\geq 0$, the object
  $\xi\mbox{-}{\rm pd}K_{i+1}''<\infty$, thus the induced sequence
  $$
  0\to \mathcal{C}(G,K_{i+1}'')\to \mathcal{C}(G,L_i)\to \mathcal{C}(G,K_i'')\to 0
  $$
  is exact for any $G\in\mathcal{GP}(\xi)$. This means that the relevant $\mathbb{E}$-triangle $\xymatrix{K_{i+1}''\ar[r]&L_i\ar[r]&K_i''\ar@{-->}[r]&}$ is $\mathcal{C}(\mathcal{GP}(\xi),-)$-exact for each $i\geq 0$, and hence we obtain a $\xi$-$\mathcal{G}$projective resolution $\mathbf{G}\to M$:
  $$
  0\to G_n\to \cdots\to G_1\to G_0\to M\to 0.
  $$
  Now since the columns in the diagram (\ref{diag2}) are split $\mathbb{E}$-triangles, one has an exact sequence of complexes
  $$
  0\to \mathcal{C}(\mathbf{P},N)\to  \mathcal{C}(\mathbf{S}_{\succeq -1},N)\to  \mathcal{C}(\mathbf{G}[-1],N)\to 0.
  $$
  This shows that there is a long exact sequence
  \begin{multline*}
   0\to H^{-1} \mathcal{C}(\mathbf{S}_{\succeq -1},N)\to H^{-1}\mathcal{C}(\mathbf{G}[-1],N)\to H^0\mathcal{C}(\mathbf{P},N)\to H^0  \mathcal{C}(\mathbf{S}_{\succeq -1},N) \\
    \to H^0 \mathcal{C}(\mathbf{G}[-1],N) \to H^1 \mathcal{C}(\mathbf{P},N)\to H^1 \mathcal{C}(\mathbf{S}_{\succeq -1},N)\to H^1\mathcal{C}(\mathbf{G}[-1],N)\to \cdots \\
   \to H^{i-1} \mathcal{C}(\mathbf{G}[-1],N)\to H^i\mathcal{C}(\mathbf{P},N)\to  H^i\mathcal{C}(\mathbf{S}_{\succeq -1},N)\to  H^i\mathcal{C}(\mathbf{G}[-1],N)\to \cdots.
  \end{multline*}
  Notice that $H^{i-1} \mathcal{C}(\mathbf{G}[-1],N)=\xi{\rm xt}^i_{\mathcal{GP}(\xi)}(M,N)$, and $H^i\mathcal{C}(\mathbf{P},N)=\xi{\rm xt}^i_{\xi}(M,N)$ for any $i\geq 0$. Moreover, since $\xymatrix@C=2em{\mathbf{S}\ar[r]&\mathbf{P}\ar[r]&M}$ is a $\xi$-complete resolution of $M$, one has
  $H^i\mathcal{C}(\mathbf{S}_{\succeq -1},N)=\widetilde{\xi{\rm xt}}^i_{\mathcal{P}}(M,N)$ for any $i\geq 1$.
  Finally, by setting $K={\rm Ker}(H^0 \mathcal{C}(\mathbf{G}[-1],N) \to H^1 \mathcal{C}(\mathbf{P},N))$, we can get the desired long exact sequence.
\end{proof}

\begin{rem}\label{rem:4.5} Note that extriangulated categories are a simultaneous generalization of abelian categories and triangulated categories. It follows that  Theorem \ref{AM} here unifies Theorem 7.1 proved by Avramov and Martsinkovsky \cite{AM} in the category of modules, and Theorem 4.10 proved by Ren, Zhao and Liu \cite{RL3} in a triangulated category. It should be noted that our results here are new for  exact categories and extension-closed subcategories of triangulated categories.
\end{rem}

\begin{cor}
Let $M\in\widetilde{\mathcal{GP}}(\xi)$. Then the following are equivalent:
\begin{itemize}
  \item [(1)] $\xi$-$\mathcal{G}{\rm pd}M\leq n$.
  \item [(2)] $\xi{\rm xt}^{i}_{\mathcal{GP}(\xi)}(M,N)=0$ for all $i\geq n+1$ and all $N\in\mathcal{C}$.
  \item [(3)] The maps $\widetilde{\varepsilon}^i_{\mathcal{P}}(M,N):\xi{\rm xt}^i_{\xi}(M,N)\to \widetilde{\xi{\rm xt}}^i_{\mathcal{P}}(M,N)$ are bijective for all $i\geq n+1$ and all $N\in\mathcal{C}$.
  \item[(4)] $\xi{\rm xt}^i_{\xi}(M,Q)=0$ for all $i\geq n+1$ and all $Q\in\widetilde{\mathcal{P}}(\xi)$.
  \item[(5)] $\xi{\rm xt}^i_{\xi}(M,Q)=0$ for all $i\geq n+1$ and all $Q\in{\mathcal{P}}(\xi)$.
\end{itemize}
\end{cor}

\begin{proof}
(1) $\Leftrightarrow$ (3) follow from \cite[Proposition 3.7]{HZZZ2}, and (1) $\Leftrightarrow$ (4)  $\Leftrightarrow$ (5) follow from \cite[Theorem 3.8]{HZZ}.

(1) $\Rightarrow$ (2) is clear.

(2) $\Rightarrow$ (3) follows from Theorem \ref{AM} directly.
\end{proof}

\begin{cor}
Assume that
 $\xi\mbox{-}\mathcal{G}{\rm pd}M=n<\infty$. Then there are homomorphisms natural in $M$ and $N$,
 such that the following sequence
 \begin{multline*}
   \xymatrix@C=0.5cm{0\ar[r]&K\ar[r]&{\xi}{\rm xt}^1_{\mathcal{GP}(\xi)}(M,N)\ar[r]&{\xi}{\rm xt}^1_{\xi}(M,N)\ar[r]&\widetilde{{\xi}{\rm xt}}^1_{\mathcal{P}}(M,N)\ar[r]&{\xi}{\rm xt}^2_{\mathcal{GP}(\xi)}(M,N)\ar[r]&\cdots} \\
   \xymatrix@C=0.5cm{\cdots\ar[r]&\widetilde{{\xi}{\rm xt}}^{n-1}_{\mathcal{P}}(M,N)\ar[r]&{\xi}{\rm xt}^n_{\mathcal{GP}(\xi)}(M,N)\ar[r]&{\xi}{\rm xt}^n_{\xi}(M,N)\ar[r]&\widetilde{{\xi}{\rm xt}}^n_{\mathcal{P}}(M,N)\ar[r]&0}
 \end{multline*}
 is exact.
\end{cor}

Assume that $M\in\widetilde{\mathcal{GP}}(\xi)$ and $N\in\widetilde{\mathcal{GI}}(\xi)$. By Theorem \ref{thm:balance-of-gorenstein}, we have
$$
\xi{\rm xt}_{\mathcal{GP}(\xi)}^n(M,N)\cong \xi{\rm xt}_{\mathcal{GI}(\xi)}^n(M,N)
$$
for any $n\geq 1$, which is denoted by $\xi{\rm xt}_{\mathcal{G}(\xi)}^n(M,N)$.

By \cite[Proposition 4.3]{HZZZ2}, for any $M\in\widetilde{\mathcal{GP}}(\xi)$ and $N\in\widetilde{\mathcal{GI}}(\xi)$, we also have $$
\widetilde{\xi{\rm xt}}_{\mathcal{P}(\xi)}^n(M,N)\cong \widetilde{\xi{\rm xt}}_{\mathcal{I}(\xi)}^n(M,N)
$$
and we denote it by $\widetilde{\xi{\rm xt}}_{\xi}^n(M,N)$ for any integer $n\geq1$.

\begin{cor}
Assume that $M\in\widetilde{\mathcal{GP}}(\xi)$ and $N\in\widetilde{\mathcal{GI}}(\xi)$. Let $n=\min\{\xi\mbox{-}\mathcal{G}{\rm pd}M,\xi\mbox{-}\mathcal{G}{\rm id}N\}$. Then there are homomorphisms natural in $M$ and $N$,
 such that the following sequence
 \begin{multline*}
   \xymatrix@C=0.5cm{0\ar[r]&K\ar[r]&{\xi}{\rm xt}^1_{\mathcal{G}(\xi)}(M,N)\ar[r]&{\xi}{\rm xt}^1_{\xi}(M,N)\ar[r]&\widetilde{{\xi}{\rm xt}}^1_{\xi}(M,N)\ar[r]&{\xi}{\rm xt}^2_{\mathcal{G}(\xi)}(M,N)\ar[r]&\cdots} \\
   \xymatrix@C=0.5cm{\cdots\ar[r]&\widetilde{{\xi}{\rm xt}}^{n-1}_{\xi}(M,N)\ar[r]&{\xi}{\rm xt}^n_{\mathcal{G}(\xi)}(M,N)\ar[r]&{\xi}{\rm xt}^n_{\xi}(M,N)\ar[r]&\widetilde{{\xi}{\rm xt}}^n_{\xi}(M,N)\ar[r]&0}
 \end{multline*}
 is exact.
\end{cor}

\vspace{4mm}
\small

\hspace{-1.2em}\textbf{Jiangsheng Hu}\\
School of Mathematics and Physics, Jiangsu University of Technology,
 Changzhou 213001, China\\
E-mail: jiangshenghu@jsut.edu.cn\\[1mm]
\textbf{Dongdong Zhang}\\
Department of Mathematics, Zhejiang Normal University,
 Jinhua 321004, China\\
E-mail: zdd@zjnu.cn\\[1mm]
\textbf{Tiwei Zhao}\\
School of Mathematical Sciences, Qufu Normal University, Qufu 273165, China\\
E-mail: tiweizhao@qfnu.edu.cn\\[1mm]
\textbf{Panyue Zhou}\\
College of Mathematics, Hunan Institute of Science and Technology, Yueyang 414006, China\\
E-mail: panyuezhou@163.com

\end{document}